\documentclass[a4paper, 12pt,oneside,reqno]{amsart}
\usepackage[a4paper]{geometry}
\usepackage[utf8]{inputenc}
\usepackage[english]{babel}
\usepackage[symbol]{footmisc}
\usepackage{amssymb,amsmath,amsthm,amsfonts,xcolor,enumerate,hyperref,comment,longtable,cleveref}

\usepackage{times}
\usepackage{cite}
\usepackage{pdflscape}
\usepackage{ulem}
\usepackage[mathcal]{euscript}
\usepackage{tikz}
\usepackage{hyperref}
\usepackage{cancel} 
\usepackage{stmaryrd}
\usepackage[all]{xy}
\usepackage{pdfsync}
\usepackage{ytableau}
\usepackage{rotating}

\usetikzlibrary{arrows}

\def\pre{\mathrm{pre}}

\def\Alt{\mathrm{Alt}}
\def\Var{\mathrm{Var}}
\def\Bicom{\mathrm{Bicom}}
\def\As{\mathrm{As}}
\def\Zin{\mathrm{Zin}}
\def\Pre{\mathrm{Pre}}

\def\Lie{\mathrm{Lie}}
\def\Leib{\mathrm{Leib}}

\def\Nov{\mathrm{Nov}}
\def\com{\mathrm{com}}

\def\Ind{\mathrm{Ind}}
\def\Flex{\mathrm{Flex}}
\def\anti{\mathrm{anti}}
\def\Assosym{\mathrm{Assosym}}
\def\non{\mathrm{non}}
\def\Dend{\mathrm{Dend}}

\let\<\langle
\let\>\rangle

\def\vecotimes{\mathbin{\vec\otimes}}

\theoremstyle{definition}
\newtheorem{definition}{Definition}[section]
\newtheorem{remark}[definition]{Remark}
\newtheorem{example}[definition]{Example}

\theoremstyle{plain}
\newtheorem{lemma}[definition]{Lemma}

\newtheorem{corollary}[definition]{Corollary}
\newtheorem{theorem}[definition]{Theorem}

\begin{document}

\title{Nonsymmetric versions of binary quadratic operads}

\author{F. A. Mashurov and  B. K. Sartayev}

\address{SICM, Southern University of Science and Technology, Shenzhen, 518055, China}

\address{
SDU University, Kaskelen, Kazakhstan}

\email{f.mashurov@gmail.com}

\address{Narxoz University, Almaty, Kazakhstan}

\address{Institute of Mathematics and Mathematical Modeling, Almaty, Kazakhstan}

\email{baurjai@gmail.com}

\subjclass[2020]{17A30, 17A50}
\keywords{nonsymmetric operad, polynomial identities, associative algebra}

\thanks{.}

\begin{abstract}
In this paper, we study the white Manin product of the associative operad $\As$ with a binary quadratic operad $\Var$. We introduce the notion of a nonsymmetric version of $\Var$ and provide a criterion for determining when the operad $\As\circ\Var$ has this property. We illustrate the construction with several examples and counterexamples. Finally, for some operads admitting nonsymmetric versions, we describe their combinatorial properties.
\end{abstract}

\maketitle

\section{Introduction}\label{sec:intro}

Operads were introduced as a language for describing compositions of operations and have become a standard framework for encoding algebraic structures of many kinds \cite{May,MarklShniderStasheff,LodayVallette}. From the algebraic point of view, an operad packages families of multilinear operations together with the rules governing substitution of one operation into another. This perspective is particularly effective when one seeks to study an entire type of algebras rather than a single example, since the relevant operations and identities may then be collected into one operadic object.

More precisely, a nonsymmetric operad over a field $\mathbf{k}$ is a sequence $\mathcal{P}=\{\mathcal{P}(n)\}_{n \geq 1}$ of vector spaces endowed with composition maps
\[
\gamma_{k_1,\dots,k_n}\colon \mathcal{P}(n) \otimes \mathcal{P}(k_1) \otimes \cdots \otimes \mathcal{P}(k_n) \to \mathcal{P}(k_1+\cdots+k_n)
\]
and a unit element $\mathbf{1} \in \mathcal{P}(1)$ satisfying the usual associativity and unit axioms. A symmetric operad is defined in the same way, except that each space $\mathcal{P}(n)$ is moreover equipped with a compatible action of the symmetric group $S_n$.

A basic reason for the usefulness of operads is that each operad governs a category of algebras of a fixed type, so results proved at the operadic level apply uniformly to all algebras in that category. The corresponding notion of a free operad is analogous to that of a free algebra: it is generated by operations subject only to operadic composition, and presentations by generators and relations arise as quotients of free operads. These ideas are particularly effective for binary quadratic operads, which occupy a central place in algebraic operad theory and in the study of Manin products \cite{GinzburgKapranov,LodayVallette}. In recent years, the white Manin product has been used to solve several natural problems, such as constructing embeddings \cite{KSO2019, KMS2024} and determining whether an operad satisfies the Dong property \cite{Dong}.

In this paper, we consider the white Manin product $\mathrm{As} \circ \mathrm{Var}$, where $\mathrm{As}$ is the associative operad and $\mathrm{Var}$ is a binary quadratic operad. The product with $\mathrm{As}$ naturally separates the symmetry of a binary generator into two oriented operations, which suggests the problem of whether the resulting symmetric operad arises from an underlying nonsymmetric one. This leads to the notion of a nonsymmetric version of $\mathrm{Var}$: roughly speaking, one asks whether $\mathrm{As} \circ \mathrm{Var}$ recovers $\mathrm{Var}$ after identifying those binary operations that differ only by the residual transposition symmetry. The case, when $\mathrm{Var}=\mathrm{Nov}$ is Novikov operad, one obtains the class of noncommutative Novikov algebras. These algebras were investigated in detail by Sartayev and Kolesnikov \cite{erlagol2021} as an embedding of noncommutative Novikov algebras into associative algebras with a derivation $d$ as follows:
\[
\non\com\textrm{-}\Nov\hookrightarrow \As^{(d)},\;\;a\succ b=d(a)b\;\;\textrm{and}\;\;a\prec b=ad(b).
\]
Later, Dauletiyarova and Sartayev constructed a basis for the free noncommutative Novikov algebra \cite{DauSar}. For the other results on noncommutative Novikov algebras, see \cite{NonComNov2, NonComNov1}.

In the case $\Var=\Zin$ operad of the variety of Zinbiel algebras, then we will see that the nonsymmetric version of the Zinbiel operad coincides with the dendriform operad. The dendriform operad was initially introduced by Loday in \cite{Loday}. Before the construction of embeddings of noncommutative Novikov algebras, Aguiar gave an embedding of dendriform algebras into associative algebras with a Rota-Baxter operator $R$ as follows:
\[
\Dend\hookrightarrow\As^{(R)},\;\;a\succ b=R(a)b\;\;\textrm{and}\;\;a\prec b=aR(b).
\]
For more details on such embedding, see \cite{Aguiar, GubarevKolesnikov2013}.

However, as we will see, not every binary quadratic operad admits a nonsymmetric version. Thus, the notion of a nonsymmetric version separates binary quadratic operads into two classes: those that satisfy the required condition and those that do not. For example, in addition to the cases discussed above, the bicommutative and flexible operads belong to the first class. In contrast, well-known operads such as the alternative, pre-Lie, and Leibniz operads belong to the second one.

In general, it is interesting to consider the mapping given above for operads that admit a nonsymmetric version, i.e., the embedding of such operads into an associative with some operator. In the case $\Var=\Bicom$, a nonsymmetric bicommutative operad can be embedded into an associative operad with operator $\mathcal{S}$ satisfying
\begin{equation}\label{bicomoperator}
a\;\mathcal{S}(\mathcal{S}(b)\;c)=\mathcal{S}(a\;\mathcal{S}(b))\;c.
\end{equation}
For more details on bicommutative algebras, see \cite{KMNZ2024, bicom1, bicom2, bicom3}.
The general questions can be stated as follows:
\begin{itemize}
    \item For a given binary quadratic operad $\Var$ that admits a nonsymmetric version, give a general algorithm for defining an operator $\mathcal{S}$ from the defining identities of $\Var$ that allows to construct the mapping from $\As\circ\Var$ to $\As^{(\mathcal{S})}$ as follows:
\[
\As\circ\Var\rightarrow\As^{(\mathcal{S})},\;\;a\succ b=\mathcal{S}(a)b\;\;\textrm{and}\;\;a\prec b=a\mathcal{S}(b);
\]
    \item Is it always possible to embed an operad $\As\circ\Var$ into $\As^{(\mathcal{S})}$?
\end{itemize}
In the case $\Var=\Nov$, the operator $\mathcal{S}$ is a derivation, and the answer to the second question is positive. Also, if $\Var=\Zin$, then for $\mathcal{S}$ stands the Rota-Baxter operator, and the answer to the second question is also positive. If $\Var=\Bicom$, it is straightforward that the operator $\mathcal{S}$ satisfies the condition \eqref{bicomoperator}. The answer to the second question is open.

All the written observations make sense that the white Manin product of the associative operad with $\Var$ generalizes the existing theory of the nonsymmetric binary quadratic operads and with a known set of operators.

This question is natural from both algebraic and combinatorial points of view. Algebraically, it measures the extent to which the defining relations of $\mathrm{Var}$ are already visible before the symmetric group actions are imposed. Combinatorially, the existence of a nonsymmetric version often makes it possible to work with planar tree monomials, rewriting rules, and explicit linear bases, thereby providing concrete descriptions that are less accessible in the symmetric setting. Thus the passage to a nonsymmetric model is not merely a reformulation; it can also simplify the effective study of the operad \cite{BrD_Companion,DotsHij}. For background on nonsymmetric operads from a combinatorial point of view we also refer to \cite{Giraudo}.

\section{White Manin products of the operad As with Var}

We start with a brief recollection of the operadic language used below. Let $\Var$ be a multilinear variety of algebras, that is, a class of vector spaces equipped with multilinear operations, not necessarily binary, subject to a family of multilinear identities in the variables
\[
X=\{x_1,x_2,\ldots\}.
\]
The multilinear part of the free algebra $\Var\langle X\rangle$ naturally carries the structure of an operad, which we again denote by $\Var$. More precisely, for every $n\geq 1$, the space $\Var(n)$ consists of all multilinear elements of degree $n$ depending on the variables $x_1,\ldots,x_n$. The symmetric group $S_n$ acts on $\Var(n)$ by permutations of variables, and the operadic compositions are defined by substitutions followed by consecutive renumbering of variables. The unit of this operad is
\[
1_{\Var}=x_1\in \Var(1).
\]

For example, if $\Var=\As$ is the variety of associative algebras, then $\As(n)$ is spanned by the monomials
\[
x_{\sigma(1)}\cdots x_{\sigma(n)},\qquad \sigma\in S_n.
\]
For instance, the composition
\[
\gamma_{2,1,3}^6:
x_3x_2x_1\otimes
x_2x_1\otimes x_1\otimes x_1x_3x_2
\]
is given by
\[
x_3x_2x_1
\circ
(x_2x_1,\;x_1,\;x_1x_3x_2)
=
(x_4x_6x_5)x_3(x_2x_1)
=
x_4x_6x_5x_3x_2x_1.
\]

We also recall the standard definition of a binary quadratic operad. Such operads encode varieties of algebras with one or several bilinear operations satisfying multilinear identities of degree~$3$. We follow the notation of \cite{GinzburgKapranov}.

Let $V$ be a finite-dimensional vector space over a field $\Bbbk$ endowed with a linear action of the symmetric group $S_2$. Set
\[
\mathcal F_V(3)
=
\Ind_{S_2}^{S_3}(V\vecotimes V)
=
\Bbbk S_3\otimes_{\Bbbk S_2}(V\vecotimes V).
\]
If $R$ is an $S_3$-submodule of $\mathcal F_V(3)$, then we denote by
\[
\mathcal P=\mathcal P(V,R)
\]
the binary quadratic operad generated by $\mathcal P(2)=V$ with defining relations $R$. Thus
\[
\mathcal P(1)=\Bbbk,
\qquad
\mathcal P(3)=\mathcal F_V(3)/R.
\]

Let $\Var$ be a quadratic operad generated by two binary operations $e_1$ and $e_2$ satisfying
\[
e_2=(12)e_1.
\]
Then the white Manin product $\As\circ\Var$ naturally splits this $S_2$-symmetry and produces an operad with two pairs of binary operations. For the definition and computation of the white Manin product, we refer the reader to \cite{GinzburgKapranov}. This motivates the following definition.

\begin{definition}
Assume that the operad $\As\circ\Var$ is generated by four binary operations
\[
g_1,\qquad g_2=(12)g_1,\qquad h_1,\qquad h_2=(12)h_1.
\]
We say that $\As\circ\Var$ is a nonsymmetric version of $\Var$ if the original operad $\Var$ is recovered from $\As\circ\Var$ by the identification
\[
g_1=h_2\;\;\textrm{and}\;\;g_2=h_1
\]
Equivalently,
\[
\As\circ\Var \big/ \{\,g_1-h_2=0,\,g_2-h_1=0\}\cong \Var.
\]
In this case, we denote the nonsymmetric operad $\As\circ\Var$ by $\non\com\text{-}\Var$.
\end{definition}

Let us present several examples of nonsymmetric versions of operads $\Var$.

\begin{example}\label{ex1}
Consider $\Var=\Nov$, where $\Nov$ is the operad governed by the variety of Novikov algebras defined by the right-commutative and left-symmetric identities.

As given in \cite{KSO2019}, the operad $\As\circ\Nov$ is generated by a $4$-dimensional space with basis
\[
g_1,\qquad g_2=(12)g_1,\qquad h_1,\qquad h_2=(12)h_1,
\]
and by an $S_3$-submodule $R$ generated by the relations
\[
(13)(g_2\vecotimes h_2)-h_1\vecotimes g_1,
\]
\[
g_1\vecotimes h_1
-
(13)(g_2\vecotimes g_2)
-
(13)(h_2\vecotimes g_2)
+
h_1\vecotimes h_1.
\]
Equivalently, introducing binary operations $\succ$ and $\prec$ by
\[
x_1\succ x_2 \leftrightarrow g_1,
\qquad
x_1\prec x_2 \leftrightarrow h_1,
\]
these relations take the form
\[
x\succ (y\prec z) = (x\succ y)\prec z,
\]
\[
(x\prec y)\succ z - x\succ (y\succ z)
=
x\prec (y\succ z) - (x\prec y)\prec z.
\]
In \cite{erlagol2021}, the operad $\As\circ\Nov$ is called the operad of noncommutative Novikov algebras. Moreover, it was noted in \cite{DauSar,erlagol2021} that
\[
\non\com\text{-}\Nov \big/ \{\,g_1-h_2=0,\,g_2-h_1=0\}\cong \Nov.
\]
\end{example}

\begin{example}\label{ex2}
The operad $\Zin$ governed by the variety of Zinbiel algebras is defined by a $2$-dimensional space $V$
with a basis
\[
e_1,\qquad e_2=(12)e_1
\]
and by the $S_3$-submodule $R$ generated by
\begin{equation}\label{eq:Zin-operadic-correct}
(13)(e_2\vecotimes e_2)-e_1\vecotimes e_1-e_1\vecotimes e_2.
\end{equation}
Indeed, \eqref{eq:Zin-operadic-correct} is exactly the operadic form of the identity
\[
x_1(x_2x_3)=(x_1x_2)x_3+(x_2x_1)x_3.
\]

Similarly, the operad $\As\circ\Zin$ is defined by a $4$-dimensional space $W$
with a basis
\[
g_1,\qquad g_2=(12)g_1,\qquad h_1,\qquad h_2=(12)h_1
\]
and by the $S_3$-submodule $R$ generated by
\begin{equation}\label{eq:AsZin-rel-1-correct}
g_1\vecotimes g_1-(13)(g_2\vecotimes g_2)-(13)(g_2\vecotimes h_2),
\end{equation}
\begin{equation}\label{eq:AsZin-rel-2-correct}
g_1\vecotimes h_1-(13)(h_2\vecotimes g_2),
\end{equation}
\begin{equation}\label{eq:AsZin-rel-3-correct}
(13)(h_2\vecotimes h_2)-h_1\vecotimes g_1-h_1\vecotimes h_1.
\end{equation}
These relations correspond respectively to
\[
(x_1\succ x_2)\succ x_3=x_1\succ(x_2\succ x_3)+x_1\succ(x_2\prec x_3),
\]
\[
(x_1\prec x_2)\succ x_3=x_1\prec(x_2\succ x_3),
\]
\[
x_1\prec(x_2\prec x_3)=(x_1\succ x_2)\prec x_3+(x_1\prec x_2)\prec x_3.
\]

Now factor by the relations
\[
g_1-h_2=0\;\;\textrm{and}\;\;g_2-h_1=0.
\]
Denote by
\[
e_1=\overline{h}_1,\qquad e_2=\overline{g}_1=(12)e_1
\]
the images of the generators in the quotient. Then \eqref{eq:AsZin-rel-1-correct} and \eqref{eq:AsZin-rel-3-correct} become
\[
(13)(e_2\vecotimes e_2)-e_1\vecotimes e_2-e_1\vecotimes e_1,
\]
which is exactly the defining relation of $\Zin$. The relation \eqref{eq:AsZin-rel-2-correct} stands for left-commutative identity, which holds in $\Zin$. Hence
\[
(\As\circ\Zin)/\{g_1-h_2=0\}\cong \Zin.
\]
\end{example}

\begin{definition}
A dendriform algebra is a vector space $D$ together with bilinear operations
\[
\succeq,\preceq \colon D \times D \to D,
\]
such that the following identities hold for all $x,y,z \in D$:
\begin{align}
(x \succeq y + x \preceq y)\succeq z &= x \succeq (y \succeq z), \label{den1}\\
(x \succeq y)\preceq z &= x \succeq (y \preceq z). \label{den2}\\
(x \preceq y)\preceq z &= x \preceq (y \succeq z + y \preceq z), \label{den3}
\end{align} 

\end{definition}

\begin{remark}
Here, we obtain an observation that a nonsymmetric version of the Zinbiel operad coincides with the dendriform operad in the case
\[
a\succ b=a\preceq b\;\;\textrm{and}\;\;a\prec b=a\succeq b.
\]

Indeed, in many cases, one has
\[
\Zin\circ\Var=\pre\text{-}\Var.
\]
However, this is not true in general. In particular, if \(\Var\) is a right-nilpotent binary quadratic operad, then this property fails.
\end{remark}


\begin{example}\label{ex3}
The operad $\Bicom$ governed by the variety of bicommutative algebras is defined by a $2$-dimensional space $V$
with a basis $e_1$, $e_2=(12)e_1$ and by
the $S_3$-submodule $R$ generated by
\begin{equation}\label{eq:OperadicBicom-e1e2}
(13)(e_1\vecotimes e_2)-e_1\vecotimes e_2
\end{equation}
(the right commutativity),
\begin{equation}\label{eq:OperadicBicom-e2e1}
(13)(e_2\vecotimes e_1)-e_2\vecotimes e_1
\end{equation}
(the left commutativity).

The operad $\As\circ\Bicom$ is defined by a $4$-dimensional space $W$
with a basis
\[
g_1,\quad g_2=(12)g_1,\quad h_1,\quad h_2=(12)h_1
\]
and by the $S_3$-submodule $R$ generated by
\begin{equation}\label{eq:OperadicAsBicom-g2h2}
(13)(g_2\vecotimes h_2)-h_1\vecotimes g_1
\end{equation}
(the identity
$x_1\succ (x_2\prec x_3)=(x_1\succ x_2)\prec x_3$),
\begin{equation}\label{eq:OperadicAsBicom-h2g2}
(13)(h_2\vecotimes g_2)-g_1\vecotimes h_1
\end{equation}
(the identity
$x_1\prec (x_2\succ x_3)=(x_1\prec x_2)\succ x_3$).

Factor $\As\circ\Bicom$ by the relations
\[
g_1-h_2=0\;\;\textrm{and}\;\;g_2-h_1=0.
\]
Denote by $e_1=\overline{h}_1$ and $e_2=\overline{g}_1=(12)e_1$
the images of the generators in the quotient. Then
\eqref{eq:OperadicAsBicom-g2h2} turns into
\[
(13)(e_1\vecotimes e_2)-e_1\vecotimes e_2,
\]
and \eqref{eq:OperadicAsBicom-h2g2} turns into
\[
(13)(e_2\vecotimes e_1)-e_2\vecotimes e_1,
\]
which are exactly relations
\eqref{eq:OperadicBicom-e1e2} and
\eqref{eq:OperadicBicom-e2e1}.
Hence
\[
(\As\circ\Bicom)/\{g_1-h_2=0,\,g_2-h_1=0\}\cong \Bicom.
\]
\end{example}

However, not every binary quadratic operad admits a nonsymmetric version. The following result illustrates this.

\begin{lemma}\label{NoAlt}
The operad $\Alt$ does not admit a nonsymmetric version, where $\Alt$
corresponds to the variety of alternative algebras \cite{AKK2026, binaryperm}. More precisely,
\[
\As\circ\Alt \big/ \{\,g_1-h_2=0,\,g_2-h_1=0\}\not\cong \Alt.
\]
\end{lemma}

\begin{proof}
The operad $\As\circ\Alt$ is defined by a $4$-dimensional space $W$
with a basis
\[
g_1,\quad g_2=(12)g_1,\quad h_1,\quad h_2=(12)h_1
\]
and by the $S_3$-submodule $R$ generated by
\begin{equation}\label{eq:OperadicAsAlt-g1h1}
g_1\vecotimes g_1
-(13)(g_2\vecotimes g_2)
-h_1\vecotimes h_1
+(13)(h_2\vecotimes h_2)
\end{equation}
(the identity
\[
(x_1 \succ x_2)\succ x_3 - x_1 \succ (x_2 \succ x_3)
-
(x_1 \prec x_2)\prec x_3 + x_1 \prec (x_2 \prec x_3)=0).
\]

Now factor $\As\circ\Alt$ by the relations
\[
g_1-h_2=0\;\;\textrm{and}\;\;g_2-h_1=0.
\]
Denote by
\[
e_1=\overline{h}_1,\qquad e_2=\overline{g}_1=(12)e_1
\]
the images of the generators in the quotient. Then relation
\eqref{eq:OperadicAsAlt-g1h1} turns into
\begin{equation}\label{eq:OperadicFlex-e1e2}
e_2\vecotimes e_2
-(13)(e_1\vecotimes e_1)
-e_1\vecotimes e_1
+(13)(e_2\vecotimes e_2).
\end{equation}

Let us rewrite \eqref{eq:OperadicFlex-e1e2} in terms of the binary operation
$e_1=x_1x_2$, $e_2=(12)e_1=x_2x_1$. We have
\[
e_2\vecotimes e_2=x_3(x_2x_1),\qquad
(13)(e_1\vecotimes e_1)=(x_3x_2)x_1,
\]
\[
e_1\vecotimes e_1=(x_1x_2)x_3,\qquad
(13)(e_2\vecotimes e_2)=x_1(x_2x_3).
\]
Hence \eqref{eq:OperadicFlex-e1e2} is exactly
\[
x_3(x_2x_1)-(x_3x_2)x_1-(x_1x_2)x_3+x_1(x_2x_3)=0,
\]
or, equivalently,
\[
(x_1x_2)x_3-x_1(x_2x_3)=-(x_3x_2)x_1+x_3(x_2x_1).
\]
Thus the quotient operad is defined by the identity
\[
(a,b,c)=-(c,b,a),
\]
where $(a,b,c)$ is the associator. This is precisely the flexible identity.

Therefore,
\[
(\As\circ\Alt)/\{\,g_1-h_2=0\,\}\cong\Flex,
\]
where $\Flex$ denotes the operad of flexible algebras.

It remains to note that $\Flex\not\cong \Alt$. It follows from
\[
\dim(\Alt(3))=7\;\;\textrm{and}\;\;\dim(\anti\textrm{-}\Flex(3))=9.
\]
\end{proof}

\begin{remark}\label{NoAssosym}
Similarly, one can prove that the assosymmetric operad does not admit a nonsymmetric version.  For more details on assosymmetric algebras, see, for example, \cite{Kleinfeld, MK2022}.
\end{remark}

The following result provides a criterion for determining whether a binary quadratic operad admits a nonsymmetric version.

\begin{theorem}\label{criterion}
Let $\mathcal P=\mathcal P(V,R)$ be a binary quadratic operad, and let
$\mathcal F_V(3)=\Ind_{S_2}^{S_3}(V\vecotimes V)$.
Denote by $F$ the $S_3$-submodule of $\mathcal F_V(3)$ generated by all elements from
\[
(\Bbbk\,\mathrm{id}+\Bbbk(13))(V\vecotimes V)
\]
whose images are zero in $\mathcal P(3)=\mathcal F_V(3)/R$.
Equivalently,
\[
F=\left\langle\,(\Bbbk\,\mathrm{id}+\Bbbk(13))(V\vecotimes V)\cap R\,\right\rangle_{S_3}.
\]
Then the operad $\mathcal P$ admits a nonsymmetric version if and only if
\begin{equation}\label{EqDim}
\dim (\mathcal P(V,R)(3))=\dim (\mathcal P(V,F)(3)).    
\end{equation}
\end{theorem}

Practically, the condition of the theorem means that from the space of relations $R$ one derives all consequences of the form

\begin{equation}\label{TwoOutside}
\sum\limits_{k,l\in I}
\alpha^{i,j}_{k,l} (x_3\circ_k x_2)\circ_l x_1
+
\beta^{i,j}_{k,l} (x_1\circ_k x_2)\circ_l x_3
=0,
\end{equation}
for some coefficients $\alpha^{i,j}_{k,l},\beta^{i,j}_{k,l}\in\Bbbk$.
We collect all such consequences in the set $F$.

\begin{proof}
The operad $\As\circ\mathcal{P}$ is generated by the $4$-dimensional space $(\As\circ\mathcal{P})(2)$ spanned by
\[
\begin{aligned}
x_1\prec x_2 &= x_1x_2 \otimes (x_1\circ x_2), \qquad
x_2\prec x_1 = x_2x_1 \otimes (x_2\circ x_1),\\
x_1\succ x_2 &= x_1x_2 \otimes (x_2\circ x_1), \qquad
x_2\succ x_1 = x_2x_1 \otimes (x_1\circ x_2).
\end{aligned}
\]

Since the factor $\As$ is generated by a nonsymmetric binary operation, in order to compute $\As\circ\mathcal P$ it is enough to consider the intersection of the $S_3$-submodule of $M_3(X)\otimes M_3(X)$ generated by
\[
\begin{gathered}
(x_1\prec x_2)\prec x_3=(x_1x_2)x_3\otimes (x_1\circ x_2)\circ x_3,\qquad
x_1\succ (x_2\succ x_3)=x_1(x_2x_3)\otimes (x_3\circ x_2)\circ x_1,\\
(x_1\succ x_2)\prec x_3=(x_1x_2)x_3\otimes (x_2\circ x_1)\circ x_3,\qquad
x_1\succ (x_2\prec x_3)=x_1(x_2x_3)\otimes (x_2\circ x_3)\circ x_1,\\
(x_1\prec x_2)\succ x_3=(x_1x_2)x_3\otimes x_3\circ (x_1\circ x_2),\qquad
x_1\prec (x_2\succ x_3)=x_1(x_2x_3)\otimes x_1\circ (x_3\circ x_2),\\
x_1\prec (x_2\prec x_3)=x_1(x_2x_3)\otimes x_1\circ (x_2\circ x_3),\qquad
(x_1\succ x_2)\succ x_3=(x_1x_2)x_3\otimes x_3\circ (x_2\circ x_1)
\end{gathered}
\]
with the kernel of the natural projection
\[
M_3(X)\otimes M_3(X)\to \As(3)\otimes \mathcal{P}(3).
\]

Suppose that the operad $\mathcal P$ admits a nonsymmetric version. Then, by definition,
\[
(\As\circ\mathcal P)\big/\{\,x_1\succ x_2-x_2\prec x_1=0\,\}\cong \mathcal P.
\]
Hence, every degree-$3$ monomial in the generators of $\As\circ\mathcal P$ reduces to a linear combination of monomials of the form
\[
(x_{i_1}\prec x_{i_2})\prec x_{i_3},\qquad x_{i_1}\prec(x_{i_2}\prec x_{i_3}),
\]
that is, to elements of
\[
(\Bbbk\,\mathrm{id}+\Bbbk(13))(V\vecotimes V)\subseteq \mathcal F_V(3).
\]
Therefore, after factorization by the relation \(x_1\succ x_2-x_2\prec x_1=0\), all defining relations of \(\As\circ\mathcal P\) become consequences of \(R\) of the form
\[
\sum\limits_{k,l\in I}
\alpha^{i,j}_{k,l}(x_3\circ_k x_2)\circ_l x_1
+
\beta^{i,j}_{k,l}(x_1\circ_k x_2)\circ_l x_3=0.
\]
By definition, the $S_3$-submodule generated by all such consequences is precisely \(F\). Thus,
\[
(\As\circ\mathcal P)\big/\{\,x_1\succ x_2-x_2\prec x_1=0\,\}\cong \mathcal P(V,F).
\]
Since, by assumption, this quotient is isomorphic to \(\mathcal P=\mathcal P(V,R)\), we obtain
\[
\mathcal P(V,F)\cong \mathcal P(V,R).
\]
In particular,
\[
\dim (\mathcal P(V,R)(3))=\dim (\mathcal P(V,F)(3)).
\]

Conversely, suppose that
\[
\dim(\mathcal P(V,R)(3))=\dim(\mathcal P(V,F)(3)).
\]
Since \(F\subseteq R\), we have
\[
\dim(\mathcal P(V,R)(3))=\dim( \mathcal F_V(3))-\dim R,
\qquad
\dim (\mathcal P(V,F)(3))=\dim (\mathcal F_V(3))-\dim F.
\]
Hence \(\dim R=\dim F\). Together with the inclusion \(F\subseteq R\), this implies
\[
F=R.
\]
Therefore,
\[
(\As\circ\mathcal P)\big/\{\,x_1\succ x_2-x_2\prec x_1=0\,\}\cong \mathcal P(V,F)=\mathcal P(V,R)=\mathcal P.
\]\end{proof}

\begin{remark}
The operads considered in Examples~\ref{ex2} and~\ref{ex3} satisfy the condition of Theorem~\ref{criterion}. Indeed, the defining identities of $\Zin$ and $\Bicom$ are already written in the form \eqref{TwoOutside}. Namely, for $\Zin$ one has
\[
x_1(x_2x_3)-(x_1x_2)x_3+(x_2x_1)x_3=0,
\]
and for $\Bicom$ one has
\[
(x_2x_1)x_3-(x_2x_3)x_1=0,
\qquad
x_1(x_3x_2)-x_3(x_1x_2)=0.
\]
\end{remark}

\begin{example}\label{NovAndLS}
Let $\Pre\text{-}\Lie$ be the operad defined by the left-symmetric identity
\[
(x_1x_2)x_3-x_1(x_2x_3)-(x_2x_1)x_3+x_2(x_1x_3)=0.
\]
By Theorem~\ref{criterion}, the operad $\Pre\text{-}\Lie$ does not admit a nonsymmetric version, since its defining identity cannot be written in the form \eqref{TwoOutside}.

On the other hand, the defining identities of the operad $\Nov$ may be written as
\[
(x_1x_2)x_3-x_1(x_3x_2)-(x_3x_2)x_1+x_3(x_1x_2)=0
\]
and
\[
(x_2x_1)x_3-(x_2x_3)x_1=0.
\]
Hence, they are of the form \eqref{TwoOutside}, and therefore $\Nov$ admits a nonsymmetric version.
\end{example} 

\begin{lemma}
Consider $\Var=\Leib$ defined by the identity
\[
f:=(x_1x_2)x_3-x_1(x_2x_3)+x_2(x_1x_3)=0.
\]
The $\Leib$ does not admit a nonsymmetric version.
\end{lemma}

\begin{proof}
Let $V=\Bbbk e_1\oplus \Bbbk e_2$, where $e_2=(12)e_1$, and let $e_1=x_1x_2,\;e_2=x_2x_1$.
Then the operad $\Leib=\mathcal P(V,R)$ is defined by the $S_3$-submodule $R$ generated by
\[
r=e_1\vecotimes e_1-(13)(e_2\vecotimes e_2)+(23)(e_2\vecotimes e_1).
\]
Set
\[
U_0=V\vecotimes V,\qquad U_1=(13)(V\vecotimes V),\qquad U_2=(23)(V\vecotimes V).
\]
Then
\[
\mathcal F_V(3)=U_0\oplus U_1\oplus U_2\;\;\textrm{and}\;\;
(\Bbbk\,\mathrm{id}+\Bbbk(13))(V\vecotimes V)=U_0\oplus U_1.
\]

Consider the $S_3$-orbit of \(r\):
\[
\begin{aligned}
r_1=(x_1x_2)x_3-x_1(x_2x_3)+x_2(x_1x_3),\;
r_2=(x_2x_1)x_3-x_2(x_1x_3)+x_1(x_2x_3),\\
r_3=(x_3x_2)x_1-x_3(x_2x_1)+x_2(x_3x_1),\;
r_4=(x_1x_3)x_2-x_1(x_3x_2)+x_3(x_1x_2),\\
r_5=(x_2x_3)x_1-x_2(x_3x_1)+x_3(x_2x_1),\;
r_6=(x_3x_1)x_2-x_3(x_1x_2)+x_1(x_3x_2).
\end{aligned}
\]
Direct calculations show that
\[
R\cap (U_0\oplus U_1)=\left\langle r_1+r_2,\ r_3+r_5\right\rangle .
\]
Indeed,
\[
r_1+r_2=(x_1x_2)x_3+(x_2x_1)x_3\;\;\textrm{and}\;\;r_3+r_5=(x_2x_3)x_1+(x_3x_2)x_1,
\]
and their $(23)$-image is
\[
(x_1x_3)x_2+(x_3x_1)x_2.
\]
Hence the space \(F\) from Theorem~\ref{criterion} is the $S_3$-submodule generated by
\[
(x_1x_2)x_3+(x_2x_1)x_3,\qquad
(x_1x_3)x_2+(x_3x_1)x_2,\qquad
(x_2x_3)x_1+(x_3x_2)x_1.
\]
It is straightforward that
\[
\dim (\mathcal P(V,F)(3))>\dim (\mathcal P(V,R)(3)).
\]
By Theorem~\ref{criterion}, the operad \(\Leib\) does not admit a nonsymmetric version.
\end{proof}

\begin{remark}
As shown in Lemma~\ref{NoAlt} and Remark~\ref{NoAssosym}, the operads $\Alt$ and $\Assosym$ do not admit nonsymmetric versions. However, if one replaces them by their generalizations, namely the flexible and anti-flexible operads, then, by Theorem~\ref{criterion} and the fact that their defining identities are of the form \eqref{TwoOutside}, these operads do admit nonsymmetric versions. 
\end{remark}


In the next sections, we consider several well-known operads that admit nonsymmetric versions. For each of these nonsymmetric versions, we construct a linear basis and describe its connection with certain combinatorial objects.

\section{Nonsymmetric Zinbiel operad}

The defining identities of the nonsymmetric version of the Zinbiel operad are given in Example \ref{ex2}.

All results of this section were proved in \cite{Loday}. Nevertheless, for the convenience of the reader, we include alternative proofs that highlight the combinatorial nature of nonsymmetric operads.

\begin{lemma}\label{GrobnerBasisZin}
For the operad $\non\com$-$\Zin$, the following set of rewriting rules forms a Gr\"obner basis:
\begin{equation}\label{GrobnerZin1}
x(*\, y(*\, *)) \;\rightarrow\; y(x(*\, *)\, *),    
\end{equation}
\begin{equation}\label{GrobnerZin2}
x(*\, x(*\, *)) \;\rightarrow\; x(y(*\, *)\, *) + x(x(*\, *)\, *),
\end{equation}
and
\begin{equation}\label{GrobnerZin3}
y(*\, y(*\, *)) \;\rightarrow\; -\,y(*\, x(*\, *)) + y(y(*\, *)\, *),
\end{equation}
where $x(*\; *)$ and $y(*\; *)$ stand for $x_1\prec x_2$ and $x_1\succ x_2$, respectively.
\end{lemma}
\begin{proof}
Fix an admissible monomial ordering on tree monomials for which the left-hand sides of
\eqref{GrobnerZin1}--\eqref{GrobnerZin3} are the leading terms. For more details on ordering and explicit computations, see \cite{BrD_Companion}.

The only overlapping compositions are
\begin{itemize}
    \item \eqref{GrobnerZin2} with \eqref{GrobnerZin2}: \(x(*\; x(*\; x(*\; *)))\);
    \item \eqref{GrobnerZin2} with \eqref{GrobnerZin1}: \(x(*\; x(*\; y(*\; *)))\);
    \item \eqref{GrobnerZin1} with \eqref{GrobnerZin3}: \(x(*\; y(*\; y(*\; *)))\);
    \item \eqref{GrobnerZin3} with \eqref{GrobnerZin3}: \(y(*\; y(*\; y(*\; *)))\).
\end{itemize}
It is straightforward that all these overlap compositions reduce to zero modulo
\eqref{GrobnerZin1}--\eqref{GrobnerZin3}. Hence, the given rewriting rules form a Gr\"obner basis.
\end{proof}

\begin{theorem}
For every \(n\ge 1\), the dimension of the \(n\)-th homogeneous component of the operad $\non\com$-$\Zin$ is equal to the \(n\)-th Catalan number, that is,
\[
\dim(\non\com\text{-}\Zin(n))=\frac{(2n)!}{n!(n+1)!}.
\]
\end{theorem}
\begin{proof}
The following two lemmas fully prove the given Theorem.

\begin{lemma}\label{NormalFormsNonComZin}
A tree monomial in $\non\com$-$\Zin$ is normal if and only if it belongs to the recursively defined class
$$
\mathcal N
=
\{\mathbf 1\}
\sqcup
x(\mathcal N,\mathbf 1)
\sqcup
y(\mathcal N,\mathbf 1)
\sqcup
y\bigl(\mathcal N, x(\mathcal N,\mathbf 1)\bigr).
$$
Equivalently, every nontrivial normal monomial has exactly one of the three forms
$$
x(u,\;\mathbf 1),\qquad
y(u,\;\mathbf 1),\qquad
y\bigl(u,\;x(v,\;\mathbf 1)\bigr),
$$
where $\mathbf 1$ denotes the unique tree monomial of arity $1$, and $u$, $v$ are normal tree monomials.
\end{lemma}

\begin{proof}
Let $T$ be a normal tree monomial. If the root of $T$ is labelled by $x$, then its right child cannot be internal. Indeed, if the right child were internal, then its root would be either $x$ or $y$, and $T$ would contain one of the forbidden divisors
$$
x(*\;x(*\;*)),\qquad x(*\;y(*\;*)).
$$
Hence
$$
T=x(u,\;\mathbf 1)
$$
for some normal tree monomial $u$. Assume now that the root of $T$ is labelled by $y$. Then its right child is either a leaf, in which case
$$
T=y(u,\;\mathbf 1),
$$
or an internal tree. In the latter case, the root of the right subtree cannot be labelled by $y$, otherwise $T$ would contain the forbidden divisor
$$
y(*\; y(*\; *)).
$$
Therefore the right subtree is rooted at $x$. Since an $x$-rooted normal subtree cannot have an internal right child, it must be of the form
$$
x(v,\;\mathbf 1).
$$
Hence
$$
T=y\bigl(u,\;x(v,\;\mathbf 1)\bigr)
$$
for some normal tree monomials $u$ and $v$.

Conversely, every tree monomial described by the recursive rule above avoids the three forbidden divisors
$$
x(*\;y(*\;*)),\qquad x(*\;x(*\;*)),\qquad y(*\;y(*\;*)).
$$
Thus it is normal.
\end{proof}

\begin{lemma}\label{BijectionNonComZin}
For every $n\ge 1$, there is a bijection between the set $\mathcal N_n$ of normal tree monomials of arity $n$ in $\non\com$-$\Zin$ and the set $\mathcal B_n$ of planar binary trees with $n$ internal vertices.
\end{lemma}

\begin{proof}
We define the bijection recursively.

For $n=1$, the set $\mathcal N_1$ consists of the single degenerate tree $\mathbf 1$, and $\mathcal B_1$ consists of the unique planar binary tree with one internal vertex. So the claim is clear.
Let $T\in\mathcal N_n$, where $n\ge 2$. By Lemma~\ref{NormalFormsNonComZin}, $T$ has exactly one of the following forms:
$$
T=x(u,\mathbf 1),\qquad
T=y(u,\mathbf 1),\qquad
T=y\bigl(u,x(v,\mathbf 1)\bigr).
$$

We associate with $T$ a planar binary tree $\Phi(T)$ as follows. If
$$T=x(u,\mathbf 1),$$
then $\Phi(T)$ is the planar binary tree whose root has left subtree $\Phi(u)$ and right child a leaf. If
$$T=y(u,\mathbf 1),$$
then $\Phi(T)$ is the planar binary tree whose root has left child a leaf and right subtree $\Phi(u)$. If
$$T=y\bigl(u,x(v,\mathbf 1)\bigr),$$
then $\Phi(T)$ is the planar binary tree whose root has left subtree $\Phi(v)$ and right subtree $\Phi(u)$. This gives a well-defined planar binary tree with $n$ internal vertices.

Conversely, let $B\in\mathcal B_n$. If $n=1$, set $\Phi^{-1}(B)=\mathbf 1$. If $n\ge 2$, inspect the two children of the root of $B$. If the right child is a leaf, write
$$
B=\operatorname{node}(L,\bullet)
$$
and set
$$
\Phi^{-1}(B)=x\bigl(\Phi^{-1}(L),\mathbf 1\bigr).
$$
If the left child is a leaf, write
$$
B=\operatorname{node}(\bullet,R)
$$
and set
$$
\Phi^{-1}(B)=y\bigl(\Phi^{-1}(R),\mathbf 1\bigr).
$$
If both children are internal, write
$$
B=\operatorname{node}(L,R)
$$
and set
$$
\Phi^{-1}(B)=y\bigl(\Phi^{-1}(R),x(\Phi^{-1}(L),\mathbf 1)\bigr).
$$
It is straightforward to verify that these constructions are mutually inverse.
\end{proof}

So, by Lemma \ref{GrobnerBasisZin}, Lemma \ref{NormalFormsNonComZin} and Lemma \ref{BijectionNonComZin},  we obtain
\[
\dim(\non\com\text{-}\Zin(n))=|\mathcal N_n|=|\mathcal B_n|=C_n.
\]

\end{proof}

In particular, for the operad $\non\com\textrm{-}\Zin$, we have
\[
\begin{array}{c|cccccccccc}
 n & 1 & 2 & 3 & 4 & 5 & 6 & 7 & 8 & 9 & 10\\ \hline
 \dim(\non\com\textrm{-}\Zin(n)) & 1 & 2 & 5 & 14 & 42 & 132 & 429 & 1430 & 4862 & 16796
\end{array}
\]

Let us describe the bijection of Lemma~\ref{BijectionNonComZin} explicitly for $n=3,4$, and $5$.

\begin{example}
For convenience, we represent a planar binary tree by a full parenthesization of the symbol $\bullet$. Thus, for example,
$$
((\bullet\bullet)\bullet)
$$
denotes the planar binary tree whose root has left subtree $(\bullet\bullet)$ and right child a leaf.

\medskip

For $n=3$, the set $\mathcal N_3$ consists of the following five normal monomials:
$$
x(x(\mathbf 1,\mathbf 1),\mathbf 1),\qquad
x(y(\mathbf 1,\mathbf 1),\mathbf 1),\qquad
y(x(\mathbf 1,\mathbf 1),\mathbf 1),\qquad
y(y(\mathbf 1,\mathbf 1),\mathbf 1),\qquad
y(\mathbf 1,x(\mathbf 1,\mathbf 1)).
$$
The bijection $\Phi_3:\mathcal N_3\to\mathcal B_3$ is given by
$$
x(x(\mathbf 1,\mathbf 1),\mathbf 1)
\longleftrightarrow
(((\bullet\bullet)\bullet)\bullet),\;
x(y(\mathbf 1,\mathbf 1),\mathbf 1)
\longleftrightarrow
((\bullet(\bullet\bullet))\bullet),
$$
$$
y(x(\mathbf 1,\mathbf 1),\mathbf 1)
\longleftrightarrow
(\bullet((\bullet\bullet)\bullet)),\;
y(y(\mathbf 1,\mathbf 1),\mathbf 1)
\longleftrightarrow
(\bullet(\bullet(\bullet\bullet))),
$$
$$
y(\mathbf 1,x(\mathbf 1,\mathbf 1))
\longleftrightarrow
((\bullet\bullet)(\bullet\bullet)).
$$

\medskip

For $n=4$, the bijection $\Phi_4:\mathcal N_4\to\mathcal B_4$ is given by
$$
x(x(x(\mathbf 1,\mathbf 1),\mathbf 1),\mathbf 1)
\;\leftrightarrow\;
((((\bullet\bullet)\bullet)\bullet)\bullet),\;x(x(y(\mathbf 1,\mathbf 1),\mathbf 1),\mathbf 1)
\;\leftrightarrow\;
(((\bullet(\bullet\bullet))\bullet)\bullet),
$$
$$
x(y(x(\mathbf 1,\mathbf 1),\mathbf 1),\mathbf 1)
\;\leftrightarrow\;
((\bullet((\bullet\bullet)\bullet))\bullet),
\;
x(y(y(\mathbf 1,\mathbf 1),\mathbf 1),\mathbf 1)
\;\leftrightarrow\;
((\bullet(\bullet(\bullet\bullet)))\bullet),
$$
$$
x(y(\mathbf 1,x(\mathbf 1,\mathbf 1)),\mathbf 1)
\;\leftrightarrow\;
(((\bullet\bullet)(\bullet\bullet))\bullet),
\;
y(x(x(\mathbf 1,\mathbf 1),\mathbf 1),\mathbf 1)
\;\leftrightarrow\;
(\bullet(((\bullet\bullet)\bullet)\bullet)),
$$
$$
y(x(y(\mathbf 1,\mathbf 1),\mathbf 1),\mathbf 1)
\;\leftrightarrow\;
(\bullet((\bullet(\bullet\bullet))\bullet)),
\;
y(y(x(\mathbf 1,\mathbf 1),\mathbf 1),\mathbf 1)
\;\leftrightarrow\;
(\bullet(\bullet((\bullet\bullet)\bullet))),
$$
$$
y(y(y(\mathbf 1,\mathbf 1),\mathbf 1),\mathbf 1)
\;\leftrightarrow\;
(\bullet(\bullet(\bullet(\bullet\bullet)))),
\;
y(y(\mathbf 1,x(\mathbf 1,\mathbf 1)),\mathbf 1)
\;\leftrightarrow\;
(\bullet((\bullet\bullet)(\bullet\bullet))),
$$
$$
y(\mathbf 1,x(x(\mathbf 1,\mathbf 1),\mathbf 1))
\;\leftrightarrow\;
((((\bullet\bullet)\bullet)(\bullet\bullet))),
\;
y(\mathbf 1,x(y(\mathbf 1,\mathbf 1),\mathbf 1))
\;\leftrightarrow\;
(((\bullet(\bullet\bullet))(\bullet\bullet))),
$$
$$
y(x(\mathbf 1,\mathbf 1),x(\mathbf 1,\mathbf 1))
\;\leftrightarrow\;
((\bullet\bullet)((\bullet\bullet)\bullet)),
\;
y(y(\mathbf 1,\mathbf 1),x(\mathbf 1,\mathbf 1))
\;\leftrightarrow\;
((\bullet\bullet)(\bullet(\bullet\bullet))).
$$

\medskip

For $n=5$, it is convenient to use the fourteen monomials from $\mathcal N_4$ listed above.
Denote them by
$$
U_1,\dots,U_{14},
$$
and denote their images in $\mathcal B_4$ by
$$
T_1,\dots,T_{14}.
$$
Then the bijection $\Phi_5:\mathcal N_5\to\mathcal B_5$ is described as follows.

First, there are fourteen monomials of the form
$$
x(U_i,\mathbf 1),\qquad i=1,\dots,14,
$$
and they correspond to the trees
$$
(T_i\;\bullet),\qquad i=1,\dots,14.
$$

Second, there are fourteen monomials of the form
$$
y(U_i,\mathbf 1),\qquad i=1,\dots,14,
$$
and they correspond to the trees
$$
(\bullet\; T_i),\qquad i=1,\dots,14.
$$

Finally, the bijection for the remaining fourteen monomials is
$$
y(\mathbf 1,x(x(x(\mathbf 1,\mathbf 1),\mathbf 1),\mathbf 1))
\;\leftrightarrow\;
((((\bullet\bullet)\bullet)\bullet)(\bullet\bullet)),
\qquad
y(\mathbf 1,x(x(y(\mathbf 1,\mathbf 1),\mathbf 1),\mathbf 1))
\;\leftrightarrow\;
(((\bullet(\bullet\bullet))\bullet)(\bullet\bullet)),
$$
$$
y(\mathbf 1,x(y(x(\mathbf 1,\mathbf 1),\mathbf 1),\mathbf 1))
\;\leftrightarrow\;
((\bullet((\bullet\bullet)\bullet))(\bullet\bullet)),
\qquad
y(\mathbf 1,x(y(y(\mathbf 1,\mathbf 1),\mathbf 1),\mathbf 1))
\;\leftrightarrow\;
((\bullet(\bullet(\bullet\bullet)))(\bullet\bullet)),
$$
$$
y(\mathbf 1,x(y(\mathbf 1,x(\mathbf 1,\mathbf 1)),\mathbf 1))
\;\leftrightarrow\;
(((\bullet\bullet)(\bullet\bullet))(\bullet\bullet)),
\qquad
y(x(\mathbf 1,\mathbf 1),x(x(\mathbf 1,\mathbf 1),\mathbf 1))
\;\leftrightarrow\;
(((\bullet\bullet)\bullet)((\bullet\bullet)\bullet)),
$$
$$
y(x(\mathbf 1,\mathbf 1),x(y(\mathbf 1,\mathbf 1),\mathbf 1))
\;\leftrightarrow\;
((\bullet(\bullet\bullet))((\bullet\bullet)\bullet)),
\qquad
y(y(\mathbf 1,\mathbf 1),x(x(\mathbf 1,\mathbf 1),\mathbf 1))
\;\leftrightarrow\;
(((\bullet\bullet)\bullet)(\bullet(\bullet\bullet))),
$$
$$
y(y(\mathbf 1,\mathbf 1),x(y(\mathbf 1,\mathbf 1),\mathbf 1))
\;\leftrightarrow\;
((\bullet(\bullet\bullet))(\bullet(\bullet\bullet))),
\qquad
y(x(x(\mathbf 1,\mathbf 1),\mathbf 1),x(\mathbf 1,\mathbf 1))
\;\leftrightarrow\;
((\bullet\bullet)(((\bullet\bullet)\bullet)\bullet)),
$$
$$
y(x(y(\mathbf 1,\mathbf 1),\mathbf 1),x(\mathbf 1,\mathbf 1))
\;\leftrightarrow\;
((\bullet\bullet)((\bullet(\bullet\bullet))\bullet)),
\qquad
y(y(x(\mathbf 1,\mathbf 1),\mathbf 1),x(\mathbf 1,\mathbf 1))
\;\leftrightarrow\;
((\bullet\bullet)(\bullet((\bullet\bullet)\bullet))),
$$
$$
y(y(y(\mathbf 1,\mathbf 1),\mathbf 1),x(\mathbf 1,\mathbf 1))
\;\leftrightarrow\;
((\bullet\bullet)(\bullet(\bullet(\bullet\bullet)))),
\qquad
y(y(\mathbf 1,x(\mathbf 1,\mathbf 1)),x(\mathbf 1,\mathbf 1))
\;\leftrightarrow\;
((\bullet\bullet)((\bullet\bullet)(\bullet\bullet))).
$$
\end{example}

\section{Nonsymmetric bicommutative operad}

The defining identities of the nonsymmetric bicommutative operad are given in Example \ref{ex3}.
\begin{lemma}\label{IDNonComBicom}
For any $n\geq 0$, in $\non\com$-$\Bicom$ the following rewriting rules hold:    
\begin{equation}\label{bicom1}
x\bigl(*\, \underbrace{x(x(\cdots x}_{n\text{ times}}(y(*\, *)\, *)\cdots )\, *)\bigr)
\;\rightarrow\;
\underbrace{x(x(\cdots x}_{n\text{ times}}(y(x(*\, *)\, *)\, *)\cdots )\,*)
\end{equation}
and
\begin{equation}\label{bicom2}
y\bigl(*\, \underbrace{y(y(\cdots y}_{n\text{ times}}(x(*\, *)\, *)\cdots )\, *)\bigr)
\;\rightarrow\;
\underbrace{y(y(\cdots y}_{n\text{ times}}(x(y(*\, *)\, *)\, *)\cdots )\,*),
\end{equation}
where $x(*\; *)$ and $y(*\; *)$ stand for $x_1\prec x_2$ and $x_1\succ x_2$, respectively.
\end{lemma}
\begin{proof}
We prove these identities by induction. The base of induction is $n=0$, which coincides with defining identities $\non\com$-$\Bicom$.

We prove \eqref{bicom1} and \eqref{bicom2} by using the inductive hypothesis as follows:
\begin{multline*}
x\bigl(*\, \underbrace{x(x(\cdots x}_{n\text{ times}}(y(*\, *)\, *)\cdots )\, *)\bigr)=x\bigl(*\, \underbrace{x(x(\cdots y}_{n-1\text{ times}}(*\, x(*\, *))\cdots )\, *)\bigr)=\\
\underbrace{x(x(\cdots x}_{n-1\text{ times}}(y(x(*\,*)\, x(*\, *)) \, *)\cdots )\, *)=\underbrace{x(x(\cdots x}_{n\text{ times}}(y(x(*\, *)\, *)\, *)\cdots )\,*)
\end{multline*}
and
\begin{multline*}
y\bigl(*\, \underbrace{y(y(\cdots y}_{n\text{ times}}(x(*\, *)\, *)\cdots )\, *)\bigr)=y\bigl(*\, \underbrace{y(y(\cdots x}_{n-1\text{ times}}(*\, y(*\, *))\cdots )\, *)\bigr)=\\
\underbrace{y(y(\cdots y}_{n-1\text{ times}}(x(y(*\, *)\, y(*\, *))\,*)\cdots )\, *)=\underbrace{y(y(\cdots y}_{n\text{ times}}(x(y(*\, *)\, *)\, *)\cdots )\,*).
\end{multline*}

\end{proof}

\begin{theorem}\label{GroebnerNonComBicom}
The family of rewriting rules of the form \eqref{bicom1} and \eqref{bicom2} forms a Gr\"obner basis.
\end{theorem}

\begin{proof}
Fix an admissible monomial ordering on tree monomials for which the left-hand sides of \eqref{bicom1} and \eqref{bicom2} are the leading terms. 
For details on the ordering and on explicit computations, see \cite{BrD_Companion}.

For every $n\ge 0$, denote by $f_n$ and $g_n$ the rewriting rules \eqref{bicom1} and \eqref{bicom2} of degree $n+3$, respectively, and let
\[
S=\{f_n,g_n\mid n\ge 0\}.
\]

We first consider the initial compositions. The compositions of $f_0$ and $g_0$ at the monomials
\[
x(*\; y(*\; x(*\; *)))
\qquad\text{and}\qquad
y(*\; x(*\; y(*\; *)))
\]
produce precisely the rules $f_1$ and $g_1$, respectively. Hence they are trivial modulo $S$.

Next, the composition of $g_1$ with $f_0$ at
\[
y(*\; y(x(*\; y(*\; *))\; *))
\]
gives exactly the rule $g_2$, and the composition of $f_1$ with $g_0$ at
\[
x(*\; x(y(*\; x(*\; *))\; *))
\]
gives exactly the rule $f_2$.

In general, for every $n\ge 0$, the compositions of the pairs $(g_n,f_0)$ and $(f_n,g_0)$ at the monomials
\[
y\bigl(*\, y(y(\cdots x(*\, y(*\, *))\cdots )\, *)\bigr)
\]
and
\[
x\bigl(*\, x(x(\cdots y(*\, x(*\, *))\cdots )\, *)\bigr)
\]
produce precisely the rules $g_{n+1}$ and $f_{n+1}$, respectively. Therefore all these compositions are trivial modulo $S$.

It remains to verify that all other overlap compositions of rules from $S$ are also trivial. These computations are straightforward, although lengthy, and may be carried out directly or checked with the aid of computer algebra software; see \cite{DotsHij}. Thus all compositions are trivial modulo $S$, and therefore the family $S$ forms a Gr\"obner basis.
\end{proof}

\begin{remark}
Another way to prove Theorem~\ref{GroebnerNonComBicom} is to construct the multiplication table for a spanning set of the operad $\non\com\text{-}\Bicom$. The spanning monomials of $\non\com\text{-}\Bicom$ may be derived from Lemma~\ref{IDNonComBicom}.
\end{remark}

\begin{theorem}
For every \(n\ge 1\), the dimension of the \(n\)-th homogeneous component of the operad $\non\com$-$\Bicom$ is given by the \(n\)-th central binomial coefficient, that is,
\[
\dim(\non\com\text{-}\Bicom(n))=\frac{(2n-2)!}{(n-1)!\,(n-1)!}.
\]
\end{theorem}
\begin{proof}
To prove the result, we first prove several supporting lemmas.
\begin{lemma}\label{NormalFormsNonComBicom}
Let
\[
\mathcal X=\{\mathbf 1\}\sqcup x(\mathcal X,\mathcal X),
\qquad
\mathcal Y=\{\mathbf 1\}\sqcup y(\mathcal Y,\mathcal Y).
\]
Thus $\mathcal X$ consists of all tree monomials whose internal vertices are labelled only by $x$, and $\mathcal Y$ consists of all tree monomials whose internal vertices are labelled only by $y$.

A tree monomial in $\non\com$-$\Bicom$ is normal if and only if it belongs to the recursively defined class
\[
\mathcal N=\{\mathbf 1\}
\sqcup
x(\mathcal N,\mathcal X)
\sqcup
y(\mathcal N,\mathcal Y).
\]
Equivalently, every nontrivial normal monomial has exactly one of the two forms
\[
x(u,\alpha),\qquad y(u,\beta),
\]
where $u$ is a normal tree monomial, $\alpha\in\mathcal X$, and $\beta\in\mathcal Y$.
\end{lemma}

\begin{proof}
We prove it by induction on the arity. For arity $1$, the only tree monomial is $\mathbf 1$, and it is normal. Assume that the statement holds for all arities less than $n$, and let $T$ be a normal tree monomial of arity $n$.

Suppose first that the root of $T$ is labelled by $x$, so that
\[
T=x(u,v),
\]
and we claim that $v\in\mathcal X$.

If $v=\mathbf 1$, there is nothing to prove. Assume that $v$ is internal. Then the root of $v$ cannot be labelled by $y$, otherwise $T$ would contain the forbidden divisor
\[
x(*\;y(*\;*)),
\]
which is the case $n=0$ of \eqref{bicom1}. Hence
\[
v=x(v_1,v_2).
\]

The subtree $v$ is itself normal, and its arity is strictly smaller than $n$. Therefore, by the induction hypothesis applied to the normal tree $v$, its right subtree satisfies
\[
v_2\in\mathcal X.
\]

Next, we show that $v_1\in\mathcal X$. If $v_1=y(w_1,x_2)$, then
\[
T=x(u,x(y(w_1,w_2),v_2))
\]
has the forbidden form by $n=1$ of \eqref{bicom1}. Therefore
\[
v=x(v_1,v_2)\in\mathcal X,\;\;\textrm{and}\;\; T=x(u,\alpha) \qquad\text{with}\qquad u\in\mathcal N,\ \alpha\in\mathcal X.
\]

The argument for a tree whose root is labelled by $y$ is completely symmetric. If
\[
T=y(u,v),
\]
then \(v\) must belong to \(\mathcal Y\). Otherwise, \(T\) would contain a forbidden divisor of type \eqref{bicom2}. Thus, every nontrivial normal monomial has one of the required two forms.

Conversely, let $T$ belong to the recursively defined class
\[
\mathcal N
=
\{\mathbf 1\}
\sqcup
x(\mathcal N,\mathcal X)
\sqcup
y(\mathcal N,\mathcal Y).
\]
If \(T=\mathbf 1\), then it is normal. Suppose
\[
T=x(u,\alpha),
\qquad
u\in\mathcal N,\ \alpha\in\mathcal X.
\]
By induction, \(u\) is normal. Since \(\alpha\) has no internal vertex labelled by \(y\), no forbidden divisor of type \eqref{bicom1} can involve the root of \(T\). Moreover, no forbidden divisor can lie inside \(\alpha\), because all internal vertices of \(\alpha\) are labelled by \(x\). Hence \(T\) is normal. The case
\[
T=y(u,\beta),
\qquad
u\in\mathcal N,\ \beta\in\mathcal Y,
\]
is analogical.

\end{proof}

\begin{lemma}\label{BijectionNonComBicom}
For every \(n\ge 1\), there is a bijection between the set \(\mathcal N_n\) of normal tree monomials of arity \(n\) in \(\non\com\)-\(\Bicom\) and the set \(\mathcal W_{n-1}\) of words in the alphabet \(\{E,N\}\) containing exactly \(n-1\) letters \(E\) and \(n-1\) letters \(N\). Equivalently, there is a bijection between \(\mathcal N_n\) and the set of lattice paths from \((0,0)\) to \((n-1,n-1)\) with steps \(E=(1,0)\) and \(N=(0,1)\).
\end{lemma}

\begin{proof}
For \(m\ge 0\), denote by \(\mathcal D_m^{+}\subset \mathcal W_m\) the set of Dyck words of semilength \(m\), i.e., the words \(w\in\mathcal W_m\) such that every prefix of \(w\) contains at least as many letters \(E\) as letters \(N\). Similarly, denote by \(\mathcal D_m^{-}\subset \mathcal W_m\) the set of reflected Dyck words, i.e., the words \(w\in\mathcal W_m\) such that every prefix of \(w\) contains at least as many letters \(N\) as letters \(E\).

Firstly, we define bijections
\[
\delta_x:\mathcal X\longrightarrow \mathcal D^{+}:= \bigsqcup_{m\ge 0}\mathcal D_m^{+},
\qquad
\delta_y:\mathcal Y\longrightarrow \mathcal D^{-}:= \bigsqcup_{m\ge 0}\mathcal D_m^{-},
\]
recursively by
\[
\delta_x(\mathbf 1)=\varepsilon,
\qquad
\delta_x\bigl(x(\alpha_1,\alpha_2)\bigr)
=
E\,\delta_x(\alpha_1)\,N\,\delta_x(\alpha_2),
\]
and
\[
\delta_y(\mathbf 1)=\varepsilon,
\qquad
\delta_y\bigl(y(\beta_1,\beta_2)\bigr)
=
N\,\delta_y(\beta_1)\,E\,\delta_y(\beta_2),
\]
where \(\varepsilon\) denotes the empty word. These are the standard inverse bijections defined recursively by
\[
\delta_x^{-1}(\varepsilon)=\mathbf 1,
\qquad
\delta_x^{-1}(E\,p\,N\,q)
=
x\bigl(\delta_x^{-1}(p),\delta_x^{-1}(q)\bigr),
\]
where \(E\,p\,N\) is the first-return decomposition of a Dyck word. Symmetrically,
\[
\delta_y^{-1}(\varepsilon)=\mathbf 1,
\qquad
\delta_y^{-1}(N\,p\,E\,q)
=
y\bigl(\delta_y^{-1}(p),\delta_y^{-1}(q)\bigr).
\]

Now, we define the mapping in general form as follows:
\[
\Phi:\mathcal N\longrightarrow \bigsqcup_{m\ge 0}\mathcal W_m
\]
recursively by
\[
\Phi(\mathbf 1)=\varepsilon,
\]
and, for every normal monomial, we define
\[
\Phi\bigl(x(u,\alpha)\bigr)
=
\Phi(u)\,E\,\delta_x(\alpha)\,N,
\]
\[
\Phi\bigl(y(u,\beta)\bigr)
=
\Phi(u)\,N\,\delta_y(\beta)\,E,
\]
where \(u\in\mathcal N\), \(\alpha\in\mathcal X\), and \(\beta\in\mathcal Y\), as in Lemma~\ref{NormalFormsNonComBicom}.

Let us construct the inverse map. Let \(w\in\mathcal W_{n-1}\). Decompose \(w\) uniquely at its returns to the diagonal:
\[
w=w_1w_2\cdots w_k,
\]
where each \(w_i\) is a nonempty balanced word having no proper nonempty balanced prefix. Then each \(w_i\) is necessarily of one of the two forms
\[
w_i=E\,d_i\,N
\qquad\text{with}\qquad
d_i\in\mathcal D^{+},
\]
or
\[
w_i=N\,e_i\,E
\qquad\text{with}\qquad
e_i\in\mathcal D^{-}.
\]
Indeed, each factor \(w_i\) starts on the diagonal and returns to the diagonal for the first time at its endpoint, so between its endpoints it stays entirely on one side of the diagonal.

Starting with
\[
T_0=\mathbf 1,
\]
define recursively
\[
T_i=
x\bigl(T_{i-1},\delta_x^{-1}(d_i)\bigr)
\quad\text{if } w_i=E\,d_i\,N,
\]
and
\[
T_i=
y\bigl(T_{i-1},\delta_y^{-1}(e_i)\bigr)
\quad\text{if } w_i=N\,e_i\,E.
\]
By Lemma~\ref{NormalFormsNonComBicom}, each \(T_i\) is a normal tree monomial. Moreover, a straightforward induction on \(i\) shows that
\[
\Phi(T_i)=w_1w_2\cdots w_i.
\]
In particular,
\[
\Phi(T_k)=w.
\]
Thus we obtain an explicit inverse map
\[
\Phi^{-1}:\mathcal W_{n-1}\longrightarrow \mathcal N_n.
\]

\end{proof}

So, by Theorem~\ref{GroebnerNonComBicom}, Lemma~\ref{NormalFormsNonComBicom}, and Lemma~\ref{BijectionNonComBicom}, we obtain
\[
\dim(\non\com\text{-}\Bicom(n))=|\mathcal N_n|=|\mathcal W_{n-1}|=\binom{2n-2}{n-1}.
\]
    
\end{proof}

In particular, for the operad $\non\com\textrm{-}\Bicom$, we have
\[
\begin{array}{c|cccccccccc}
 n & 1 & 2 & 3 & 4 & 5 & 6 & 7 & 8 & 9 & 10\\ \hline
 \dim(\non\com\textrm{-}\Bicom(n)) & 1 &  2 &  6 &  20 &  70 &  252 &  924 &  3432 &  12870 &  48620
\end{array}
\]

\begin{example}
For \(n=3\), the bijection \(\Phi_3:\mathcal N_3\to\mathcal W_2\) is given by
$$
x(\mathbf 1,x(\mathbf 1,\mathbf 1))
\;\leftrightarrow\;
ENEN,
\qquad
x(x(\mathbf 1,\mathbf 1),\mathbf 1)
\;\leftrightarrow\;
EENN,
$$
$$
y(x(\mathbf 1,\mathbf 1),\mathbf 1)
\;\leftrightarrow\;
ENNE,
\qquad
x(y(\mathbf 1,\mathbf 1),\mathbf 1)
\;\leftrightarrow\;
NEEN,
$$
$$
y(y(\mathbf 1,\mathbf 1),\mathbf 1)
\;\leftrightarrow\;
NNEE,
\qquad
y(\mathbf 1,y(\mathbf 1,\mathbf 1))
\;\leftrightarrow\;
NENE.
$$

Let us demonstrate the bijections $\Phi_4:\mathcal N_4\to\mathcal W_3$ and $\Phi_5:\mathcal N_5\to\mathcal W_4$ explicitly.

For $n=4$, the bijection $\Phi_4$ is given by
$$
x(\mathbf 1,x(x(\mathbf 1,\mathbf 1),\mathbf 1))
\;\leftrightarrow\;
ENEENN,\qquad
x(\mathbf 1,x(\mathbf 1,x(\mathbf 1,\mathbf 1)))
\;\leftrightarrow\;
ENENEN.
$$
$$
x(x(\mathbf 1,x(\mathbf 1,\mathbf 1)),\mathbf 1)
\;\leftrightarrow\;
EENENN,\qquad
y(x(\mathbf 1,x(\mathbf 1,\mathbf 1)),\mathbf 1)
\;\leftrightarrow\;
ENENNE.
$$
$$
x(x(\mathbf 1,\mathbf 1),x(\mathbf 1,\mathbf 1))
\;\leftrightarrow\;
EENNEN,\qquad
x(x(x(\mathbf 1,\mathbf 1),\mathbf 1),\mathbf 1)
\;\leftrightarrow\;
EEENNN.
$$
$$
y(x(x(\mathbf 1,\mathbf 1),\mathbf 1),\mathbf 1)
\;\leftrightarrow\;
EENNNE,\qquad
x(y(x(\mathbf 1,\mathbf 1),\mathbf 1),\mathbf 1)
\;\leftrightarrow\;
ENNEEN.
$$
$$
y(y(x(\mathbf 1,\mathbf 1),\mathbf 1),\mathbf 1)
\;\leftrightarrow\;
ENNENE,\qquad
y(x(\mathbf 1,\mathbf 1),y(\mathbf 1,\mathbf 1))
\;\leftrightarrow\;
ENNNEE.
$$
$$
x(y(\mathbf 1,\mathbf 1),x(\mathbf 1,\mathbf 1))
\;\leftrightarrow\;
NEEENN,\qquad
x(x(y(\mathbf 1,\mathbf 1),\mathbf 1),\mathbf 1)
\;\leftrightarrow\;
NEENEN.
$$
$$
y(x(y(\mathbf 1,\mathbf 1),\mathbf 1),\mathbf 1)
\;\leftrightarrow\;
NEENNE,\qquad
x(y(y(\mathbf 1,\mathbf 1),\mathbf 1),\mathbf 1)
\;\leftrightarrow\;
NNEEEN.
$$
$$
y(y(y(\mathbf 1,\mathbf 1),\mathbf 1),\mathbf 1)
\;\leftrightarrow\;
NNNEEE,\qquad
y(y(\mathbf 1,\mathbf 1),y(\mathbf 1,\mathbf 1))
\;\leftrightarrow\;
NNEENE.
$$
$$
x(y(\mathbf 1,y(\mathbf 1,\mathbf 1)),\mathbf 1)
\;\leftrightarrow\;
NENEEN,\qquad
y(y(\mathbf 1,y(\mathbf 1,\mathbf 1)),\mathbf 1)
\;\leftrightarrow\;
NNENEE.
$$
$$
y(\mathbf 1,y(\mathbf 1,y(\mathbf 1,\mathbf 1)))
\;\leftrightarrow\;
NENENE,\qquad
y(\mathbf 1,y(y(\mathbf 1,\mathbf 1),\mathbf 1))
\;\leftrightarrow\;
NENNEE.
$$

\end{example}

\section{Nonsymmetric (anti)-flexible operad}

The defining identity of the nonsymmetric flexible operad is given in Lemma \ref{NoAlt}.

\begin{lemma}\label{basisNonComFlex}
For the operad $\non\com$-$\Flex$, the following rewriting system forms a Gr\"obner basis:
\begin{equation}\label{NonComFlexID1}
    y(*\, y(*\, *))
\;\rightarrow\;
\,x(*\, x(*\, *))
+ y(y(*\, *)\, *)
- x(x(*\, *)\, *),
\end{equation}
and
\begin{multline}\label{NonComFlexID2}
y(*\, x(*\, x(*\, *)))
\;\rightarrow\;\;
y(*\, x(x(*\, *)\, *))
+ x(*\, x(*\, y(*\, *))) - x(*\, x(y(*\, *)\, *))
- y(x(*\, x(*\, *))\, *) \\
+ x(y(*\, *)\, x(*\, *))
- x(x(*\, *)\, y(*\, *)) + x(x(*\, y(*\, *))\, *)
+ y(x(x(*\, *)\, *)\, *) - x(x(y(*\, *)\, *)\, *) .
\end{multline}
Here \(x(*\, *)\) and \(y(*\, *)\) stand for \(x_1\prec x_2\) and \(x_1\succ x_2\), respectively.
\end{lemma}
\begin{proof}
Fix an admissible monomial ordering on tree monomials for which the left-hand side of
\eqref{NonComFlexID1} is the leading term.

The first nontrivial overlap composition is
\begin{itemize}
    \item \eqref{NonComFlexID1} with \eqref{NonComFlexID1}: $y(*\; y(*\; y(*\; *)))$,
\end{itemize}
which gives exactly \eqref{NonComFlexID2}.
The next overlapping composition is
\begin{itemize}
    \item \eqref{NonComFlexID1} with \eqref{NonComFlexID2}: $y(*\; y(*\; x(*\; x(*\; *))))$.
\end{itemize}

It is straightforward that all these overlap compositions reduce to zero modulo
\eqref{NonComFlexID1} and \eqref{NonComFlexID2}. Hence, the given rewriting rules form a Gr\"obner basis.

\end{proof}

\begin{theorem}\label{DimNonComFlex}
For every \(n\ge 1\), the dimension of the \(n\)-th homogeneous component of the operad $\non\com$-$\Flex$ is given by
\[
a(n):=\frac{1}{n}\binom{3n-2}{n-1}.
\]
Equivalently,
\[
\dim(\non\com\text{-}\Flex(n))
=
\frac{(3n-2)!}{n!\,(2n-1)!}.
\]
\end{theorem}
\begin{proof}
Firstly, let us define the set of normal tree monomials for \(\non\com\)-\(\Flex\). By Theorem~\ref{basisNonComFlex}, the normal tree monomials in \(\non\com\)-\(\Flex\) are precisely the tree monomials which avoid the leading terms
\[
F_1:=y(*\,y(*\,*)),
\qquad
F_2:=y(*\,x(*\,x(*\,*))).
\]

Let \(\mathcal N\) be the class of all normal tree monomials, and let \(\mathcal R\subset \mathcal N\) be the class of those normal tree monomials which may occur as the right subtree of a root labelled by \(y\), without creating at the root either \(F_1\) or \(F_2\).

We claim that
\begin{equation}\label{RecNFlex1}
\mathcal N
=
\{\mathbf 1\}
\sqcup
x(\mathcal N,\mathcal N)
\sqcup
y(\mathcal N,\mathcal R),
\end{equation}
and
\begin{equation}\label{RecNFlex2}
\mathcal R
=
\{\mathbf 1\}
\sqcup
x\bigl(\mathcal N,\{\mathbf 1\}\sqcup y(\mathcal N,\mathcal R)\bigr).
\end{equation}

Indeed, let \(T\in\mathcal N\) be nontrivial. If the root of \(T\) is labelled by \(x\), then no forbidden divisor can occur at the root, since both \(F_1\) and \(F_2\) start with \(y\). Hence both subtrees of the root are arbitrary normal tree monomials, and \(T\in x(\mathcal N,\mathcal N)\).

Suppose that the root of \(T\) is labelled by \(y\), and let \(r\) be its right subtree. Then \(r\) cannot have root \(y\), otherwise \(T\) would contain \(F_1\) at the root. Thus \(r\) is either \(\mathbf 1\), or \(r=x(u,v)\). In the latter case, the subtree \(v\) cannot have root \(x\), otherwise \(T\) would contain \(F_2\) at the root. Therefore
\[
v\in \{\mathbf 1\}\sqcup y(\mathcal N,\mathcal R),
\]
and hence \(r\in\mathcal R\). This proves \eqref{RecNFlex1} and \eqref{RecNFlex2}. The converse inclusion is immediate by induction on the arity.

In \cite{Leroux}, it was shown that the dimension of the \(L\)-operad defined by the identity
\[
z(t(*\; *)\; *)=t(*\;z(*\;*))
\]
is equal to \(a(n)\).

It is straightforward to verify that the rewriting rule
\[
t(*\;z(*\;*))\rightarrow z(t(*\; *)\; *)
\]
forms a Gr\"obner basis for the \(L\)-operad. Therefore, the set \(\mathcal S\) of normal tree monomials in the \(L\)-operad is given by
\[
\mathcal{S}=\{\mathbf 1\}
\sqcup
z(\mathcal S,\mathcal S)
\sqcup
t(\mathcal S,t(\mathcal S,t(\cdots t(\mathcal S,\mathcal S)\cdots))).
\]

On the other hand, the set of normal tree monomials in \(\non\com\)-\(\Flex\), described by \eqref{RecNFlex1} and \eqref{RecNFlex2}, may be written in the form
\[
\mathcal{N}=\{\mathbf 1\}
\sqcup
x(\mathcal N,\mathcal N)
\sqcup
y(\mathcal N,x(\mathcal N,y(\mathcal N,x(\cdots)))).
\]
Hence, there is a natural bijection between \(\mathcal S\) and \(\mathcal N\), given as follows:

Let
\[
\mathcal U:=\{\mathbf 1\}\sqcup t(\mathcal S,\mathcal U),
\qquad
\mathcal Q:=\{\mathbf 1\}\sqcup y(\mathcal N,\mathcal R).
\]
Then
\[
\mathcal S=\{\mathbf 1\}\sqcup z(\mathcal S,\mathcal S)\sqcup t(\mathcal S,\mathcal U),
\]
\[
\mathcal N=\{\mathbf 1\}\sqcup x(\mathcal N,\mathcal N)\sqcup y(\mathcal N,\mathcal R),
\qquad
\mathcal R=\{\mathbf 1\}\sqcup x(\mathcal N,\mathcal Q).
\]

There is a bijection between \(\mathcal S\) and \(\mathcal N\), defined simultaneously with bijections
\[
\Phi_R:\mathcal U\to\mathcal R,
\qquad
\Phi_Q:\mathcal U\to\mathcal Q,
\]
as follows:
\[
\Phi(\mathbf 1)=\mathbf 1,
\qquad
\Phi_R(\mathbf 1)=\mathbf 1,
\qquad
\Phi_Q(\mathbf 1)=\mathbf 1,
\]
\[
\Phi\bigl(z(u,v)\bigr)=x\bigl(\Phi(u),\Phi(v)\bigr),
\]
\[
\Phi\bigl(t(u,v)\bigr)=y\bigl(\Phi(u),\Phi_R(v)\bigr),
\]
\[
\Phi_R\bigl(t(u,v)\bigr)=x\bigl(\Phi(u),\Phi_Q(v)\bigr),
\]
\[
\Phi_Q\bigl(t(u,v)\bigr)=y\bigl(\Phi(u),\Phi_R(v)\bigr).
\]
In other words, the operation \(z\) is sent to \(x\), and every right comb of \(t\)'s is sent to the alternating right spine of \(x\)'s and \(y\)'s in \(\non\com\)-\(\Flex\).

The inverse maps are defined recursively by
\[
\Psi(\mathbf 1)=\mathbf 1,\qquad
\Psi\bigl(x(u,v)\bigr)=z(\Psi(u),\Psi(v)),
\]
\[
\Psi\bigl(y(u,r)\bigr)=t(\Psi(u),\Psi_R(r)),
\]
\[
\Psi_R(\mathbf 1)=\mathbf 1,\qquad
\Psi_R\bigl(x(u,q)\bigr)=t(\Psi(u),\Psi_Q(q)),
\]
\[
\Psi_Q(\mathbf 1)=\mathbf 1,\qquad
\Psi_Q\bigl(y(u,r)\bigr)=t(\Psi(u),\Psi_R(r)).
\]
It is straightforward to verify that these maps are mutually inverse.

Therefore
\[
\dim(\non\com\text{-}\Flex(n))=|\mathcal N_n|=|\mathcal S_n|=a_n.
\]

\end{proof}

One of the combinatorial objects counted by the sequence \(a(n)\), namely OEIS A006013, is the following.
\begin{corollary}
For every \(n\ge 1\), there is a bijection between the set \(\mathcal N_n\) of normal tree monomials of arity \(n\) in \(\non\com\)-\(\Flex\) and the set
\[
\bigsqcup_{i+j=n-1}\mathcal T_i\times \mathcal T_j,
\]
where \(\mathcal T_m\) denotes the set of plane rooted ternary trees with exactly \(m\) internal vertices.
\end{corollary}

In particular, for the operad $\non\com\textrm{-}\Flex$, we have
\[
\begin{array}{c|cccccccccc}
 n & 1 & 2 & 3 & 4 & 5 & 6 & 7 & 8 & 9 & 10\\ \hline
 \dim(\non\com\textrm{-}\Flex(n)) & 1 &  2 &  7 &  30 &  143 &  728 &  3876 &  21318 &  120175 &  690690
\end{array}
\]

\begin{remark}
Analogues of Lemma~\ref{basisNonComFlex} and Theorem~\ref{DimNonComFlex} also hold for the nonsymmetric version of the anti-flexible operad. More precisely, the statement of Theorem~\ref{DimNonComFlex} remains unchanged, whereas in Lemma~\ref{basisNonComFlex} one only needs to reverse the sign in the defining identity.
\end{remark}


\begin{thebibliography}{99}

\bibitem{AKK2026}
H. Abdelwahab, I. Kaygorodov, A. Khudoyberdiyev, The algebraic and geometric classification of right alternative superalgebras, Rendiconti del Circolo Matematico di Palermo, 2026, 75(3), 82.

\bibitem{Aguiar}
M. Aguiar, Pre-Poisson algebras, Letters in Mathematical Physics, 2000, 54(4), 263–277.

\bibitem{BrD_Companion}
M. R. Bremner, V. Dotsenko, 
Algebraic Operads An Algorithmic Companion, Chapman Hall, New York, 2016.

\bibitem{DauSar}
A. Dauletiyarova, B. K. Sartayev, Basis of the free noncommutative Novikov algebra, Journal of Algebra and its Applications, 2025, 24(12), 2550292.


\bibitem{DotsHij}
V. Dotsenko, W. Heijltjes,
{Gr\"obner bases for operads}, 
\url{http://irma.math.unistra.fr/~dotsenko/Operads.html} (2019).


\bibitem{bicom1}
A. Dzhumadil'daev, N. Ismailov, Polynomial identities of bicommutative algebras, Lie and Jordan elements, Communications in Algebra, 2018, 46(12), 5241--5251.

\bibitem{bicom2}
A. Dzhumadil'daev, N. Ismailov, K. Tulenbaev, Free bicommutative algebras, Serdica Mathematical Journal, 2011, 37(1), 25--44.

\bibitem{bicom3}
N. A. Ismailov, F. A. Mashurov, B. K. Sartayev, On algebras embeddable into bicommutative algebras, Communications in Algebra, 2024, 52(11), 4778--4785.


\bibitem{NonComNov2}
X. Gao, L. Guo, Z. Han and Y. Zhang, Rota-Baxter operators, differential operators, pre-and Novikov structures on groups and Lie algebras, J. Algebra 684 (2025), 109-148.

\bibitem{GinzburgKapranov}
V.~Ginzburg, M.~Kapranov,
Koszul duality for operads,
{Duke Mathematical Journal}, \textbf{76} (1994), no.~1, 203--272.

\bibitem{Giraudo}
S.~Giraudo,
{Nonsymmetric Operads in Combinatorics},
1st ed., Springer, 2019.
DOI 10.1007/978-3-030-02074-3.


\bibitem{GubarevKolesnikov2013}
    V. Y. Gubarev, P. S. Kolesnikov, Embedding of dendriform algebras into Rota--Baxter algebras, Central European Journal of Mathematics, 2013, 11(2), 226--245.

   


\bibitem{KMNZ2024} I. Kaygorodov, F. Mashurov, T. G. Nam, Z. Zhang, Products of commutator ideals of some Lie-admissible algebras. Acta Mathematica Sinica, English Series, 2024, 40(8), 1875--1892.

 \bibitem{Kleinfeld}
E. Kleinfeld,
Assosymmetric rings,
Proceedings of the American Mathematical Society, 8 (5), 1957, 983--986.




\bibitem{KMS2024}
P. Kolesnikov, F. Mashurov, B. Sartayev, On Pre-Novikov Algebras and Derived Zinbiel Variety, Symmetry, Integrability and Geometry: Methods and Applications (SIGMA), 2024, 20, 17.


\bibitem{Dong}
P. S. Kolesnikov, B. K. Sartayev, On the Dong Property for a binary quadratic operad, Journal of Algebra, 2026, 691, 428–452.

\bibitem{KSO2019}
P. S. Kolesnikov, B. Sartayev, A. Orazgaliev, Gelfand--Dorfman algebras, derived identities, and the Manin product of operads, Journal of Algebra, 2019, 539, 260--284.

\bibitem{binaryperm}
A. Kunanbayev, B. Sartayev, Binary perm algebras and alternative algebras, Communications in Algebra, 541(1), 2026, 299-307.


\bibitem{May}
J.~P.~May,
{The Geometry of Iterated Loop Spaces},
Lecture Notes in Mathematics, Vol.~271, Springer, 1972.

\bibitem{MarklShniderStasheff}
M.~Markl, S.~Shnider, and J.~Stasheff,
{Operads in Algebra, Topology and Physics},
Mathematical Surveys and Monographs, Vol.~96, American Mathematical Society, 2002.


\bibitem{MK2022}
F. Mashurov, I. Kaygorodov,
One-generated nilpotent assosymmetric algebras,
Journal of Algebra and Its Applications, Vol.~21(2), 2022, 2250031.

\bibitem{Leroux}
P. Leroux, L-Algebras, triplicial-algebras, within an equivalence of categories motivated by graphs, Communications in Algebra, 2011, 39(8), 2661–2689.

\bibitem{LodayVallette}
J.-L.~Loday and B.~Vallette,
{Algebraic Operads},
Grundlehren der Mathematischen Wissenschaften, Vol.~346, Springer, 2012.

\bibitem{Loday}
J. L. Loday, Dialgebras, Dialgebras and related operads, Berlin, Heidelberg, Springer Berlin Heidelberg, 2002, 7-66.


\bibitem{erlagol2021}
B. Sartayev, P. Kolesnikov,
Noncommutative Novikov algebras, European Journal of Mathematics, 9(2), 35, 2023.



\bibitem{NonComNov1}
X. Wang, L. Guo, H. Zhang, General multi-Novikov algebras, multi-differential algebras and their free constructions, https://arxiv.org/abs/2603.16766.


\end{thebibliography}
\end{document}